\documentclass[a4paper,11pt]{amsart}

\usepackage{mathpazo}
\usepackage{amssymb}
\usepackage[pagebackref,
    ,pdfborder={0 0 0}
    ,urlcolor=black,a4paper,hypertexnames=false]{hyperref}
\hypersetup{pdfauthor=Clara L\"oh,pdftitle=Which finitely generated Abelian groups admit isomorphic Cayley graphs?}

\usepackage{pgf}
\usepackage{tikz}
\usetikzlibrary{decorations.pathmorphing}

\newtheorem{thm}{Theorem}[section]
\newtheorem{prop}[thm]{Proposition}
\newtheorem{lem}[thm]{Lemma}
\newtheorem{cor}[thm]{Corollary}

\newtheorem{problem}[thm]{Problem}

\theoremstyle{remark}

\newtheorem{nonexa}[thm]{(Non-)Example}
\newtheorem{notation}[thm]{Notation}
\newtheorem*{ack}{Acknowledgements}

\theoremstyle{definition}
\newtheorem{defi}[thm]{Definition}

\usepackage{amsfonts}
\newcommand{\Z}{\mathbb{Z}}

\newcommand{\R}{\mathbb{R}}
\newcommand{\N}{\mathbb{N}}

\DeclareMathOperator{\id}{id}
\DeclareMathOperator{\rk}{rk_{\Z}}
\DeclareMathOperator{\Aut}{Aut}
\DeclareMathOperator{\Cayop}{Cay}
\DeclareMathOperator{\tors}{tors}
\DeclareMathOperator{\diam}{diam}
\def\epsilon{\varepsilon}

\newcommand{\Cay}[2]{\Cayop(#1,#2)}

\def\gvertex#1{%
  \fill #1 circle(0.1);
}

\author{Clara L\"oh}
\title[Finitely generated Abelian groups with isomorphic Cayley graphs]{Which finitely generated Abelian groups admit isomorphic Cayley graphs?}
\date{\today.\ \copyright{\ C.~L\"oh 2012}, 
     MSC~2010 classification: 05C25, 05C63, 20F65\\
     \phantom{\quad} Partially supported by the \emph{Institut Mittag-Leffler} (Djursholm, Sweden)}

\begin{document}

\begin{abstract}
  We show that Cayley graphs of finitely generated Abelian groups are
  rather rigid. As a consequence we obtain that two finitely generated
  Abelian groups admit isomorphic Cayley graphs if and only if they
  have the same rank and their torsion parts have the same
  cardinality. The proof uses only elementary arguments and is
  formulated in a geometric language.
\end{abstract}

\maketitle

\section{Introduction}

Cayley graphs allow us to view groups as combinatorial and geometric
objects. For instance, one of the main objectives of geometric group
theory is to understand the relation between algebraic properties of
finitely generated groups and (large scale) geometric properties
of their Cayley graphs. On the other hand, the structure of Cayley
graphs of finite groups plays an important role in
combinatorics. 

This article shows that all Cayley graphs of finitely generated Abelian
groups are rather rigid (Theorem~\ref{mainthm}), and as a consequence
that two finitely generated Abelian groups admit isomorphic Cayley
graphs if and only if they have the same rank and their torsion parts
have the same cardinality (Corollary~\ref{maincor}).

We now describe the results in more detail. For the sake of
completeness, let us briefly recall some basic notation:

\begin{defi}[Cayley graph]
  Let $G$ be a group and let $S \subset G$ be a 
  subset 
  of~$G$. The \emph{Cayley graph of~$G$ with respect to~$S$} is the
  (unlabelled, undirected) graph~$\Cay GS$ whose vertex set is~$G$ and
  whose set of edges is given by
  \[ \bigl\{ \{g, g\cdot s\} 
     \bigm| g \in G,\ s \in (S \cup S^{-1}) \setminus \{e\}
     \bigr\}.
  \]
\end{defi}

\begin{defi}
  Two finitely generated groups~$G$ and~$G'$ \emph{admit isomorphic
    Cayley graphs} if there exist \emph{finite generating} sets~$S
  \subset G$ and $S' \subset G'$ of~$G$ and~$G'$ respectively such
  that the corresponding Cayley graphs~$\Cay GS$ and $\Cay {G'}{S'}$
  are isomorphic (as unlabelled, undirected graphs).
\end{defi}

If $G$ is a finitely generated Abelian group, then the \emph{torsion
  subgroup~$\tors G$} of~$G$, i.e., the subgroup of all elements
of~$G$ of finite order, is a finite group. Moreover, the
quotient~$G/\tors G$ is a finitely generated free Abelian group and
the rank of~$G/\tors G$ is called the \emph{rank~$\rk G$ of~$G$}. 

A finitely generated group~$G$ is a \emph{CI-group} if the following
holds~\cite{li}: Whenever $S, S' \subset G $ are symmetric finite
generating sets of~$G$ such that $\Cay GS$ and $\Cay G{S'}$ are
isomorphic, then there is a group automorphism~$\varphi \in \Aut(G)$
with~$\varphi(S) = S'$.  Based on results of Trofimov about
automorphism groups of graphs, M\"oller and
Seifter~\cite{moellerseifter} showed that all finitely generated free
Abelian groups (and more generally, finitely generated torsion-free
nilpotent groups) are CI-groups. Later, Ryabchenko~\cite{ryabchenko}
provided an elementary proof of the fact that finitely generated free
Abelian groups are CI-groups.

Even though not every finitely generated Abelian group is a
CI-group~\cite{elspasturner,li}, we will show in the following that
Cayley graphs of finitely generated Abelian groups are rigid in the
following sense:

\begin{thm}[Cayley graph rigidity]\label{mainthm}
  Let $G$ and $G'$ be finitely generated Abelian groups, and let $S
  \subset G$ and $S' \subset G'$ be finite generating sets of~$G$
  and~$G'$ respectively. If $\varphi \colon \Cay GS \longrightarrow
  \Cay {G'}{S'}$ is an isomorphism of graphs, then $\varphi$ induces 
  an affine isomorphism
  \begin{align*}
    G/\tors G & \longrightarrow G'/\tors G' \\
    [g] & \longmapsto [\varphi(g)]
  \end{align*}
  of finitely generated free Abelian groups.
\end{thm}

The proof is based on a careful analysis of sufficiently convex
geodesic lines in Cayley graphs of finitely generated Abelian
groups. Similar to Ryabchenko's arguments, the key idea is that
geodesic lines generated by ``longest'' generators satisfy a certain
uniqueness property that allows to translate between the combinatorial
structure of Cayley graphs and the algebraic structure of the underlying 
Abelian groups.

Notice however that in general not every graph automorphism of a
Cayley graph of a finitely generated Abelian group is induced from an
affine group automorphism (Example~\ref{counterexample}).

As a consequence of Theorem~\ref{mainthm} we can characterise which
finitely generated Abelian groups admit isomorphic Cayley graphs:

\begin{cor}\label{maincor}
  Two finitely generated Abelian groups admit isomorphic Cayley 
  graphs if and only if they have the same rank and their torsion 
  parts have the same cardinality.
\end{cor}

As a long-term perspective one might hope that a thorough
understanding of the combinatorics of Cayley graphs of Abelian groups
could lead to an elementary proof of quasi-isometry rigidity of
virtually Abelian groups.


This article is organised as follows: In Section~\ref{sec:geodesics}
we will study the relation between geometric and algebraic properties
of geodesics in Cayley graphs of Abelian groups, which will be the
main tool to prove Theorem~\ref{mainthm}. Section~\ref{sec:proofthm}
contains the proof of Theorem~\ref{mainthm}. In
Section~\ref{sec:proofcor} we deduce Corollary~\ref{maincor} from the
theorem. Finally, for the sake of completeness Section~\ref{sec:appx}
contains an alternative approach to detecting the parity of the
torsion part, following a discussion on
\textsf{mathoverflow.net}~\cite{mathoverflow}.

\begin{ack} 
  I would like to thank Lars Scheele for bringing the discussion on
  \textsf{mathoverflow.net}~\cite{mathoverflow} about ``Cayley graph
  equivalence'' to my attention.  Moreover, I would like to thank the
  \emph{Institut Mittag-Leffler} for its hospitality, and R\"oggi
  M\"oller, Tobias Hartnick and Anders Karlsson for interesting
  discussions.
\end{ack}

\section{Geodesic lines in finitely generated Abelian groups}\label{sec:geodesics}

Even though Theorem~\ref{mainthm} and its proof are purely
combinatorial we prefer to formulate and organise the arguments in a
geometric language, in terms of geodesics and suitable convexity
properties. In Section~\ref{subsec:geodesics} we introduce the basic
geometric language for graphs and present the basic reordering
argument for geodesics in Abelian groups, in
Section~\ref{subsec:convexgeod} and
Section~\ref{subsec:quasiconvexgeod} we study the relation between
algebraic and convexity properties of geodesic lines, and in
Section~\ref{subsec:parallelism} we list basic properties of
parallel algebraic lines in Cayley graphs.

\subsection{Graphs and geodesics}\label{subsec:geodesics}

In the present article, \emph{graphs} are unlabelled, undirected,
simple graphs. An \emph{isomorphism} between graphs~$(V,E)$ and
$(V',E')$ is a bijective map~$\varphi \colon V \longrightarrow V'$
such that for all~$u, v \in V$ we have that $\{u,v\} \in E$ if and
only if~$\{\varphi(u), \varphi(v)\} \in E'$. If $\Gamma = (V,E)$ is a
(connected) graph, then the graph structure induces a
path-metric~$d_\Gamma$ on~$V$ characterised by all edges having
length~$1$. 

For our arguments the following observation will be essential: By
definition of the graph metric, any graph isomorphism is an isometry
between the sets of vertices of the graphs in question. So anything 
that can be expressed purely in terms of metric properties of the 
underlying graphs will be preserved by graph isomorphisms.

\begin{defi}[geodesic line/segment]
  Let $\Gamma = (V,E)$ be a graph. 
  \begin{itemize}
    \item A \emph{geodesic segment} in~$\Gamma$ is a finite path
      $(v_0, \dots, v_n)$ (of vertices) in~$\Gamma$ with
      \[ d_\Gamma(v_0, v_n) = n
      \]
      (equivalently, for all~$j,k \in \{0, \dots, n\}$ one
      has~$d_\Gamma(v_j, v_k) = |j-k|$).
    \item A \emph{geodesic line} in~$\Gamma$ is an infinite path~$\gamma \colon
      \Z \longrightarrow V$ in~$\Gamma$ satisfying 
      \[ d_\Gamma(\gamma(j), \gamma(k)) = |j-k|
      \]
      for all~$j, k \in \Z$ (equivalently, all consecutive finite
      subsequences of~$\gamma$ are geodesic segments).
  \end{itemize}
\end{defi}

For the sake of readability, whenever convenient we will also use the
sequence notation~``$\gamma_j := \gamma(j)$'' for points on
$\Z$-paths~$\gamma$ in graphs.

In particular, if $G$ is a group and $S$ is a generating set of~$G$,
then $G$ inherits a metric~$d_S$ from the graph structure of~$\Cay
GS$; of course, this is nothing but the word metric on~$G$ with
respect to~$S$.

One of the main points is that in Abelian groups, commutativity of 
the group structure allows us to generate new geodesics out of old
ones by changing the order of steps (Figure~\ref{fig:reorder}):

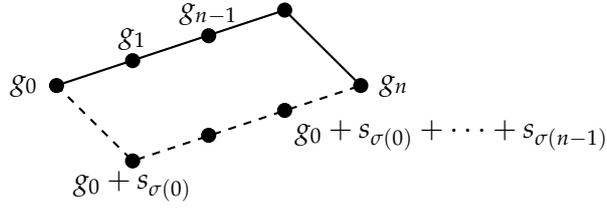
\begin{figure}
  \begin{center}
    \begin{tikzpicture}
      \draw[thick] (0,0) -- (3,1) -- (4,0);
      \draw[thick,dashed] (0,0) -- (1,-1) -- (4,0);
      \draw (0,0) node[anchor=east] {$g_0$\;};
      \draw (4,0) node[anchor=west] {\;$g_n$};
      \draw (1,0.33) node[anchor=south] {$g_1$};
      \draw (2,0.66) node[anchor=south] {$g_{n-1}$};
      \draw (1,-1) node[anchor=north] {$g_0 + s_{\sigma(0)}$};
      \draw (3,-0.33) node[anchor=north west] {$g_0 + s_{\sigma(0)} + \dots + s_{\sigma(n-1)}$};
      \gvertex{(0,0)} 
      \gvertex{(1,-1)}
      \gvertex{(2,-0.66)}
      \gvertex{(3,-0.33)}
      \gvertex{(1,0.33)}
      \gvertex{(2,0.66)}
      \gvertex{(3,1)}
      \gvertex{(4,0)}
    \end{tikzpicture}
  \end{center}
  
  \caption{Reordering geodesic segments in Abelian groups}\label{fig:reorder}
\end{figure}

\begin{prop}[reordering geodesic segments]\label{prop:reorder}
  Let $G$ be an Abelian group, let $S$ be a symmetric generating set
  of~$G$, and let $(g_0, \dots, g_n)$ be a $d_S$-geodesic
  segment in~$G$. For $j \in \{0, \dots, n-1\}$ we write
  \[ s_j := g_{j+1} - g_j. 
  \]
  Then $s_j \in S$, and for every permutation~$\sigma$ of~$\{0,\dots, n-1\}$ 
  the reordered sequence
  \[ ( g_0, g_0 + s_{\sigma(0)}, g_0 + s_{\sigma(0)} + s_{\sigma(1)}
     , \dots, g_0 + s_{\sigma(0)} + \dots + s_{\sigma(n-1)})  
  \]
  is a $d_S$-geodesic segment starting in~$g_0$ and ending in~$g_n$.
\end{prop}
\begin{proof}
  By definition of~$\Cay GS$, we have~$s_j \in S$ for all~$j \in \{0,
  \dots, n-1\}$ and the reordered sequence is a path in~$\Cay GS$. 
  By construction, the reordered path starts in~$g_0$ and ends in
  \[ g_0 + \sum_{j=0}^{n-1} s_{\sigma(j)} = g_0 + \sum_{j=0}^{n-1} s_j
      = g_0 + \sum_{j=0}^{n-1} (g_{j+1} - g_j) = g_n. 
  \]
  Moreover, because the reordered sequence has length~$n = d_S(g_0,
  g_n)$ the claim follows.
\end{proof}

\subsection{Convex geodesic lines}\label{subsec:convexgeod}

As a first step, we will give the key argument linking combinatorial
and algebraic structure of geodesic lines in a special and
straightforward case (which, e.g., is enough for the torsion-free
case).

\begin{defi}[algebraic line]
  Let $G$ be an Abelian group, let $S$ be a symmetric generating set, 
  and let $s \in S$. An \emph{algebraic line of type~$s$} in~$\Cay
  GS$ is a $\Z$-path in~$\Cay GS$ of the form
  \begin{align*}
    \Z & \longrightarrow G \\
    n & \longmapsto h + n\cdot s
  \end{align*}
  for some~$h \in G$.
\end{defi}

\begin{defi}[convex geodesic line]
  Let $\Gamma = (V,E)$ be a graph.  A \emph{convex geodesic line}
  in~$\Gamma$ is a geodesic line~$\gamma \colon \Z \longrightarrow V$
  with the following property: for all~$n,m \in \Z$ with~$n\leq m$
  there is exactly one geodesic segment in~$\Gamma$ starting
  in~$\gamma_n$ and ending in~$\gamma_m$, namely~$(\gamma_n,
  \gamma_{n+1}, \dots, \gamma_m)$.
\end{defi}

The following proposition describes the interaction between
algebraicity and convexity and shows in particular that graph
isomorphisms of Cayley graphs of finitely generated free Abelian
groups map algebraic lines of ``maximal'' type to algebraic
lines.

\begin{prop}[convex geodesic lines vs.~algebraic geodesic lines]
  \label{prop:convalg}  
  \hfil
  \begin{enumerate}
    \item If $G = \Z^r$ for some~$r \in \N$, and $S \subset G$ is a
      (symmetric) finite generating set, and $s \in S \setminus\{0\}$ is a
      $\|\cdot\|_2$-maximal element of~$S$, then all algebraic
      lines of type~$s$ in~$\Cay GS$ are convex geodesic lines.
    \item
      Graph isomorphisms map convex geodesic lines to convex geodesic lines.
    \item 
      If $G$ is an Abelian group and $S \subset G$ is a finite subset, then
      every convex geodesic line in~$\Cay GS$ is algebraic.
  \end{enumerate}
\end{prop}
\begin{proof}
  \emph{Ad~1:} Let $h \in G$. Then
  \begin{align*}
    \gamma \colon \Z & \longrightarrow G \\
    n & \longmapsto h + n \cdot s
  \end{align*}
  is a convex geodesic line in~$\Cay GS$, because: 
  Let $n, m \in \Z$ with~$m \geq n$. 
  Let $k := d_S(\gamma_n, \gamma_m)$ and let $(g_0, \dots, g_k)$ be a 
  geodesic segment in~$\Cay GS$ starting 
  in~$\gamma_n$ and ending in~$\gamma_m$. Thus, $g_{j+1} - g_j \in S$ 
  for all~$j \in \{0, \dots, k-1\}$ and
  \begin{align}\label{maxeq} 
    (m-n) \cdot s = \gamma_m - \gamma_n = \sum_{j=0}^{k-1} (g_{j+1} - g_j).
  \end{align}
  Because $s$ is $\|\cdot\|_2$-maximal in~$S$ and $s \neq 0$, it
  follows that $k \geq m-n$. Furthermore, because $(g_0, \dots, g_k)$
  is geodesic, we must have~$k = m-n$ and (again by
  $\|\cdot\|_2$-maximality of~$s$) we have $\|g_{j+1} - g_j\|_2 =
  \|s\|_2$ for all~$j \in \{0, \dots, k-1\}$. But now
  Equation~\ref{maxeq} and elementary geometry in~$\Z^r \subset \R^r$
  imply that $g_{j+1} - g_j = s$ for all~$j \in \{0, \dots, k-1\}$. Hence,
  \[ (g_0, \dots, g_m) = (h, h + s, \dots, h + n \cdot s)
     = \bigl( \gamma_n, \dots, \gamma_m\bigr),
  \]
  and so $\gamma$ is a convex geodesic  line in~$\Cay GS$.

  \emph{Ad~2:} Every graph ismorphism is an isometry with 
  respect to the graph metrics, and isometries map convex 
  geodesics to convex geodesics. This proves the second part.

  \emph{Ad~3:} Let $\gamma \colon \Z \longrightarrow G$ be a
  geodesic line in~$\Cay GS$ that is not algebraic. Then $\gamma$ is not 
  a convex geodesic line: Because $\gamma$ 
  is not algebraic, we can find~$n, m \in \Z$ with $n < m$ and
  \[ \gamma_{n +1} - \gamma_{n} \neq \gamma_{m} - \gamma_{m-1} =: s. 
  \]
  Clearly, $s \in S$ and
  because $(\gamma_{n}, \dots, \gamma_m)$ is a
  geodesic segment in~$\Cay GS$ and $G$ is Abe\-lian, also the sequence
  \[ \bigl( \gamma_n
          , \gamma_n + s
          , \gamma_{n+1} + s
          , \dots 
          , \gamma_{m-1} + s = \gamma_m
     \bigr) 
  \]
  is a geodesic segment (Proposition~\ref{prop:reorder}). However, by
  construction, this sequence does not coincide with~$(\gamma_n,
  \dots, \gamma_m)$. So, $\gamma$ is not a convex geodesic
  line. 
\end{proof}

Notice that not every algebraic line whose type is of infinite order 
is geodesic -- for example, in~$\Cay{\Z}{\{\pm 1,\pm2\}}$ algebraic lines 
of type~$1$ are \emph{not} geodesic.

\subsection{Quasi-convex geodesic lines}\label{subsec:quasiconvexgeod}

In general, torsion will introduce some ambiguities in geodesics and
we will not be able to find enough convex geodesic lines in Cayley
graphs of finitely generated Abe\-lian groups; for example, the
graph~$\Cay{\Z \times \Z/2}{\{\pm(1,0), \pm(1,1)\}}$ contains no convex
geodesic lines (even though the generating set consists of elements of
infinite order). Therefore, we introduce the slightly weaker notions 
of quasi-algebraic lines and quasi-convex geodesic lines.

\begin{notation}
  For a finitely generated Abelian group~$G$ we write
  \[ \pi_G \colon G \longrightarrow G/\tors G
  \] 
  for the canonical projection.
\end{notation}

\begin{defi}[quasi-algebraic line]
  Let $G$ be a finitely generated Abelian group, let $S \subset G$ be
  a (symmetric) generating set, and let $s \in S$. A
  \emph{quasi-algebraic line of quasi-type~$\pi_G(s)$} in~$\Cay GS$ is
  a $\Z$-path~$\gamma \colon \Z \longrightarrow G$ in $\Cay GS$ with
  the property that for all~$n \in\Z$ we have
  \[ \pi_G\bigl(\gamma_{n+1} - \gamma_n\bigr) =  \pi_G(s) \in G /\tors G. 
  \]
\end{defi}

Clearly, the quasi-type of quasi-algebraic lines is well-defined.

\begin{defi}[quasi-convex geodesic line]
  Let $\Gamma=(V,E)$ be a graph.  A \emph{quasi-convex geodesic line}
  in~$\Gamma$ is a geodesic line~$\gamma \colon \Z \longrightarrow V$
  in~$\Gamma$ such that for all~$c\in \R_{\geq 0}$ there exists a~$C \in
  \R_{\geq0}$ with the following property: All geodesic segments
  in~$\Gamma$ joining points that are $c$-close to~$\gamma$ stay
  uniformly $C$-close to~$\gamma$, i.e., for all~$n, m \in \Z$ with $n
  \leq m$, all points~$x, y \in \Gamma$ with $d_\Gamma(x,\gamma_n)
  \leq c$, $d_\Gamma(y, \gamma_m) \leq c$ and all geodesic
  segments~$\widetilde\gamma \colon \{0,\dots, d_\Gamma(x,y)\}
  \longrightarrow V$ from~$x$ to~$y$ we have for all~$j
  \in \{0, \dots, d_\Gamma(x,y)\}$ that 
  \[ d_\Gamma(\widetilde\gamma_j, \gamma_{n+j}) \leq C.
  \]
\end{defi}

Notice that quasi-convexity in geometric group theory is usually 
associated with a slightly weaker property, and that quasi-convexity 
as defined above also includes a so-called fellow-traveller property. 

In a general graph, not every geodesic line needs to be quasi-convex.
Conversely, also not all quasi-convex geodesic lines are convex: We
consider the integer square lattice $\Cay {\Z^2}{\{\pm(1,0),
  \pm(0,1)\}}$: It is not difficult to see that the sequence
\[ \dots\ , (-2,0),\ (-1,0),\ (0,0),\ (0,1),\ (1,1),\ (2,1),\ \dots
\]
is a quasi-convex geodesic line that is \emph{not} convex
(Figure~\ref{fig:notconvex}). Moreover, this is an example of a
quasi-convex geodesic line that is \emph{not} quasi-algebraic.

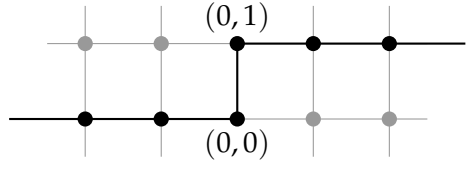
\begin{figure}
  \begin{center}
    \begin{tikzpicture}[thick]
      \begin{scope}[black!40]
        \draw[thin] (-2.5,0) -- (2.5,0);
        \draw[thin] (-2.5,1) -- (2.5,1);
        \draw[thin] (-2,-0.5) -- (-2,1.5);
        \draw[thin] (-1,-0.5) -- (-1,1.5);
        \draw[thin] (1,-0.5) -- (1,1.5);
        \draw[thin] (2,-0.5) -- (2,1.5);
        \gvertex{(-2,1)};
        \gvertex{(-1,1)};
        \gvertex{(1,0)};
        \gvertex{(2,0)};
      \end{scope}
      \gvertex{(-2,0)};
      \gvertex{(-1,0)};
      \gvertex{(0,0)};
      \gvertex{(0,1)};
      \gvertex{(1,1)};
      \gvertex{(2,1)};
      \draw[thick] (-3,0) -- (0,0) -- (0,1) -- (3,1);
      \draw (0,0) node[anchor=north] {$(0,0)$};
      \draw (0,1) node[anchor=south] {$(0,1)$};
    \end{tikzpicture}
  \end{center}

  \caption{A quasi-convex geodesic line that is neither convex nor
    quasi-algebraic (in $\Cay{\Z^2}{\{\pm(1,0),
      \pm(0,1)\}}$\,)}\label{fig:notconvex}
\end{figure}

However, we still have the following analogue of
Proposition~\ref{prop:convalg}:

\begin{prop}[quasi-convex geodesic lines vs.~quasi-algebraic lines]
  \label{prop:qconvalg}
  \hfil
  \begin{enumerate}
    \item 
      Let $G$ be a finitely generated Abelian group, let $S \subset G$ 
      be a symmetric finite generating set of~$G$, and let $\|\cdot\|_2$ be the 
      $\ell^2$-norm on~$G/\tors G$ induced by some chosen isomorphism
      $G/\tors G \cong \Z^{\rk G}$. Let $s \in S \setminus \tors G$ be 
      a $\|\pi_G(\cdot)\|_2$-maximal element of~$S$. Then all quasi-algebraic 
      lines of quasi-type~$\pi_G(s)$ in~$\Cay GS$ are quasi-convex 
      geodesic lines.
    \item Graph isomorphisms map quasi-convex geodesic lines to quasi-convex 
      geodesic lines.
    \item 
      Let $G$ and $G'$ be finitely generated Abelian groups, and
      suppose that there is an isomorphism~$\varphi \colon \Cay GS
      \longrightarrow \Cay {G'}{S'}$ for certain symmetric finite
      generating sets~$S \subset G$, $S' \subset G'$. Then $\varphi$
      maps algebraic quasi-convex geodesic lines to
      quasi-algebraic quasi-convex geodesic lines.
  \end{enumerate}
\end{prop}
\begin{proof}
  \emph{Ad~1:}
  Let $\gamma \colon \Z \longrightarrow G$ be a quasi-algebraic line of 
  quasi-type~$\pi_G(s)$. We argue similarly to the proof of the
  corresponding part of Proposition~\ref{prop:convalg}: 

  As a first step, we show that $\gamma$ is a geodesic line: To this
  end, let $n,m \in \Z$ with $n \leq m$, let $k :=
  d_S(\gamma_n,\gamma_m)$, and let $(g_0, \dots, g_k)$ be a geodesic
  segment in~$\Cay GS$ starting in~$\gamma_n$ and ending
  in~$\gamma_m$.  In particular, $g_{j+1} - g_j \in S$ for all~$j \in
  \{0, \dots, k-1\}$, and
  \begin{align*}
    (m-n) \cdot \|\pi_G(s)\|_2 
    & = \|\pi_G(\gamma_m - \gamma_n)\|_2 
      = \biggl\|\pi_G\biggl(\sum_{j=0}^{k-1} (g_{j+1} - g_j)\biggr)\biggr\|_2 
      \\
    & \leq \sum_{j=0}^{k-1} \|\pi_G(g_{j+1} - g_j)\|_2
      . 
  \end{align*}
  Because $s$ is $\|\pi_G(\cdot)\|_2$-maximal in~$S$, and 
  $\|\pi_G(s)\|_2 \neq 0$ it follows that
  \[ d_S(\gamma_n,\gamma_m) = k \geq m-n \geq d_S(\gamma_n,\gamma_m).\] 
  Hence, $(\gamma_n, \dots, \gamma_m)$ is a
  geodesic path in~$\Cay GS$, and it follows that $\gamma$ indeed is a
  geodesic line.

  Why is $\gamma$ quasi-convex?  Let $c \in \R_{\geq0}$ and let $n, m
  \in \Z$ with $n \leq m$. Moreover, let $x,y \in G$ with
  $d_S(x,\gamma_n) \leq c$ and $d_S(y, \gamma_m) \leq c$, let $k :=
  d_S(x,y)$, and let $(g_0, \dots, g_k)$ be a geodesic segment
  in~$\Cay GS$ starting in~$x$ and ending in~$y$.  In view of the
  triangle inequality and the fact that $\gamma$ is geodesic, we hence
  obtain that
  \[ m-n - 2 \cdot c \leq k \leq m - n + 2 \cdot c. 
  \]
  Let $K$ be the number of steps in the geodesic segment~$(g_0, \dots,
  g_k)$ that are \emph{not} of quasi-type~$\pi_G(s)$. We now bound $K$ 
  from above: We may assume that $K\neq 0$; then $S$ contains elements that  
  are not quasi-type~$\pm \pi_G(s)$.

  We embed~$G/\tors G$ (which we identified with~$\Z^{\rk G}$)
  into~$\R^{\rk G}$ and consider the standard scalar product
  on~$\R^{\rk G}$, which is compatible with our choice of
  $\ell^2$-norm on~$G/\tors G$. Let~$p \colon G \longrightarrow \R
  \cdot \pi_G(s) \subset \R^{\rk G}$ be the orthogonal projection onto
  the line in~$\R^{\rk G}$ spanned by~$\pi_G(s)$. Because $S$ is
  finite and $s$ is $\|\pi_G(\cdot)\|_2$-maximal in~$S$, we have
  \[ \mu := \max \bigl\{ \|p(t)\|_2 
                 \bigm| t \in S, \pi_G(t) \neq \pm \pi_G(s)
                 \bigr\}
          < \|\pi_G(s)\|_2. 
  \]
  Thus, $\|\pi_G(\cdot)\|_2$-maximality of~$s$ and the 
  triangle inequality yield
  \begin{align*}
    (m-n) \cdot \|\pi_G(s)\|_2 - 2 \cdot c \cdot \|\pi_G(s)\|_2 
    & \leq \| p(g_k) - p(g_0) \|_2 \\ 
    & \leq K \cdot \mu + (k-K) \cdot \|\pi_G(s)\|_2 \\
    & \leq K \cdot \mu + (m-n +2\cdot c-K) \cdot \|\pi_G(s)\|_2,
  \end{align*}
  and so 
  \[ K \leq \frac{4 \cdot c \cdot \|\pi_G(s)\|_2}
                 {\|\pi_G(s)\|_2 - \mu}, 
  \]
  which is the desired upper bound depending only on~$c$ and the
  geometry of~$S$.  

  Then the segment~$(g_0, \dots, g_k)$ is uniformly $C$-close
  to~$\gamma$, where
  \[ C := c +  \frac{8 \cdot c \cdot \|\pi_G(s)\|_2}
                 {\|\pi_G(s)\|_2 - \mu}
            + \diam_{d_S} \tors G
  \]
  (which depends only on~$c$ and the geometry of~$(G,S)$): If $j \in \{0,
  \dots, k\}$, then 
  \[ g_j - \gamma_{n+j} 
     \in g_0 - \gamma_n + \sum_{i = 1}^{K_j} s_i - K_j \cdot s +\tors G,
  \]
  where $s_1, \dots, s_{K_j} \in S$ are the steps in the
  segment~$(g_0, \dots, g_j)$ that are not of
  quasi-type~$\pi_G(s)$. Then $K_j \leq K$ and so
  \[ d_S(g_j, \gamma_{n+j}) \leq c + 2 \cdot K + \diam_{d_S}\tors G
     \leq C,
  \]
  as desired. 
  Hence, $\gamma$ is quasi-convex.

  \emph{Ad~2:} Every graph isomorphism is an isometry with respect to
  the graph metrics, and isometries map quasi-convex geodesic lines to
  quasi-convex geodesic lines. This proves the second part.

  \emph{Ad~3:} Let $\gamma \colon \Z \longrightarrow G$ be an
  algebraic quasi-convex geodesic line of type~$s$ for
  some~$s \in S$. By the second part, $\gamma' := \varphi
  \circ \gamma$ is a quasi-convex geodesic line in~$\Cay{G'}{S'}$. In
  order to show that $\gamma'$ is quasi-algebraic, we proceed as 
  follows:
  \begin{itemize}
    \item We will first show that all but a finite number of 
      edges in~$\gamma'$ have the same quasi-type.
    \item We will then conclude that the quasi-types of all the edges
      of~$\gamma'$ are the same.
  \end{itemize}

  For the first step, \emph{assume} for a contradiction that not all
  but a finite number of steps in~$\gamma'$ have the same
  type. Because $S'$ is finite, then there exist $s', t' \in S'$
  with~$\pi_{G'}(s') \neq \pi_{G'}(t')$ and the following property:
  for all~$k \in \N$ there is an~$n_k \in \N$ such that the geodesic
  segment~$(\gamma'_0, \dots, \gamma'_{n_k})$ contains at least
  $k$~steps of quasi-type~$\pi_{G'}(s')$ as well as at least $k$~steps of
  quasi-type~$\pi_{G'}(t')$. Because $G'$ is Abelian, the reordering 
  \begin{itemize}
    \item start in~$\gamma'_0$,
    \item then take $k$~steps of quasi-type~$\pi_{G'}(s')$,
    \item then take $k$~steps of quasi-type~$\pi_{G'}(t')$, 
    \item then take the remaining steps (ending in~$\gamma'_{n_k}$)
  \end{itemize}
  of this geodesic segment also is a geodesic segment~$\eta$ starting
  in~$\gamma'_{0}$ and ending in~$\gamma'_{n_k}$
  (Proposition~\ref{prop:reorder}). Similarly, we obtain a geodesic
  segment~$\widetilde \eta$ between~$\gamma'_0$ and~$\gamma'_{n_k}$
  that first takes $k$~steps of quasi-type~$\pi_{G'}(t')$ and then $k$~steps
  of quasi-type~$\pi_{G'}(s')$ (Figure~\ref{fig:notthin}). But then
  \[ \pi_{G'} (\eta_k - \widetilde\eta_k) 
     = k \cdot \bigl(\pi_{G'}(s') - \pi_{G'}(t')\bigr),
  \]
  and so $d_{S'}(\eta_k, \widetilde \eta_k)$ can become arbitrarily
  large (because $\pi_{G'}(s') \neq \pi_{G'}(t')$ and $\tors G'$ is
  finite).  This contradicts that $\gamma'$ is a quasi-convex geodesic
  line.

  \begin{figure}
    \begin{center}
      \begin{tikzpicture}[thick]
        \draw[black!40,decorate,decoration=zigzag] (0,0) -- (2,0);
        \draw[decorate,decoration=zigzag] (2,0) -- (4,0);
        \draw (0,0) -- (1,1) -- (2,0);
        \draw (0,0) -- (1,-1) -- (2,0);
        \gvertex{(0,0)}
        \gvertex{(1,1)}
        \gvertex{(1,-1)}
        \gvertex{(4,0)}
        \draw (3,0) node [anchor=north] {$\gamma'$};
        \draw (0.5,0.5) node[anchor= south east]{$\eta$};
        \draw (0.5,-0.5) node[anchor=north east]{$\widetilde\eta$};
        \draw (1,1) node[anchor=south west] {$\eta_k$};
        \draw (1,-1) node[anchor=north west] {$\widetilde\eta_k$};
      \end{tikzpicture}
    \end{center}
    
    \caption{Reordered geodesics that are not close}\label{fig:notthin}
  \end{figure}
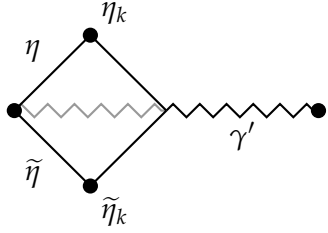

  For the second step, \emph{assume} for a contradiction that 
  not all steps of~$\gamma'$ have the same quasi-type. In view of 
  the first step and translation invariance of~$d_{S'}$ we 
  may assume that there are $s', t' \in S'$ with $\pi_{G'}(s') \neq 
  \pi_{G'}(t')$ and 
  \[ \pi_{G'}(\gamma'_{n+1} - \gamma'_n) = \pi_{G'}(s')  
  \]
  for all~$n \in \N_{>0}$, as well as
  \[ \pi_{G'}(\gamma'_1 - \gamma'_0) = \pi_{G'}(t'). 
  \]
  For $n,m \in \Z$ let $\Gamma(n,m)$ and $\Gamma'(n,m)$ denote 
  the number of geodesics in~$\Cay GS$ and $\Cay{G'}{S'}$ between 
  $\gamma_n$ and $\gamma_m$ and between $\gamma'_n$ and $\gamma'_m$ 
  respectively. Because $\gamma$ is \emph{algebraic}, because translations in~$G$ 
  are $d_S$-isometries and because $\varphi$ is an isometry with respect 
  to the word metrics~$d_S$ and $d_{S'}$, we obtain
  \[ \Gamma(0, m- n) = \Gamma(n,m) = \Gamma'(n,m)
  \]
  for all~$n,m \in \Z$. By definition of~$\Cay GS$, clearly
  \[ \Gamma(0,k) \leq |S|^k 
  \]
  for all~$k \in \N$. Hence there is a~$k \in \N$ with 
  $\Gamma(0,k+1) < k \cdot \Gamma(0,k)$. On the other hand, 
  shuffling in the step~$\gamma'_1 - \gamma'_0$ at the various 
  positions of the geodesic segments between~$\gamma'_1$ and~$\gamma'_{k+1}$
  produces geodesic segments (by Proposition~\ref{prop:reorder}), 
  which are all different because~$\pi_{G'}(t') \neq \pi_{G'}(s')$; 
  hence
  \[ (k+1) \cdot \Gamma'(1,k+1) \leq \Gamma'(0,k+1).  
  \]
  Combining these estimates, we obtain
  \begin{align*} 
    (k+1) \cdot \Gamma(0,k) 
     & = (k+1) \cdot \Gamma(1,k+1) = (k+1) \cdot \Gamma'(1,k+1)\\ 
     & \leq \Gamma'(0,k+1) = \Gamma(0,k+1)\\ 
     & < k \cdot \Gamma(0,k),
  \end{align*}
  which is a contradiction. 

  Therefore all steps of~$\gamma'$ have the same quasi-type, and 
  so $\gamma'$ is indeed a quasi-algebraic line in~$\Cay{G'}{S'}$.
\end{proof}

\subsection{Parallelism}\label{subsec:parallelism}

We now discuss two basic properties of parallel quasi-algebraic
lines. On the one hand, we show that graph isomorphisms cannot map
quasi-algebraic lines of the same quasi-type to quasi-algebraic
lines of different quasi-types; on the other hand, we show that all
quasi-algebraic lines parallel to quasi-algebraic (quasi-convex)
geodesic lines must also be (quasi-convex) geodesic.

\begin{prop}[parallel lines and quasi-type]\label{prop:paralleltype}
  Let $G$ and $G'$ be finitely generated Abelian groups, let $S
  \subset G$ and $S' \subset G'$ be symmetric finite generating sets.  Let $s
  \in S$ and let $\gamma, \eta \colon \Z \longrightarrow G$ be
  quasi-algebraic lines in~$\Cay GS$ of
  quasi-type~$\pi_G(s)$.  Suppose that there is a graph
  isomorphism~$\varphi \colon \Cay GS \longrightarrow \Cay{G'}{S'}$
  such that $\gamma' := \varphi \circ \gamma$ and $\eta' := \varphi
  \circ \eta$ are quasi-algebraic lines in~$\Cay
        {G'}{S'}$. Then $\gamma'$ and $\eta'$ have the same
        quasi-type.
\end{prop}
\begin{proof}
  Let $s', t' \in S'$ be chosen in such a way that
  $\pi_{G'}(s')$ and $\pi_{G'}(t')$ are the types of~$\gamma'$ and
  $\eta'$ respectively.  We now have to show that $\pi_{G'}(s') =
  \pi_{G'}(t')$, i.e, that $s' - t'$ lies in~$\tors
  G'$.

  In view of the definition of quasi-algebraic lines and by 
  rearranging the torsion contributions, we find 
  maps~$h, k \colon \Z \longrightarrow \tors G$,
  $h', k' \colon \Z \longrightarrow \tors G'$
  such that for all~$n\in \Z$ we have
  \begin{align*}
    \gamma_n & = \gamma_0 + n \cdot s + h_n \\
    \eta_n & = \eta_0 + n \cdot s + k_n \\
    \gamma'_n & = \gamma'_0 + n \cdot s' + h'_n \\
    \eta'_n & = \eta'_0 + n \cdot t' + k'_n. 
  \end{align*}
  Because $\varphi$ and all translations are isometries we obtain for
  all~$n \in \Z$ that
  \begin{align*}
    d_{S'}\bigl(n \cdot (s' - t'),\; &\ \eta'_0-\gamma'_0 + k'_n - h'_n\bigr)\\
    & =  d_{S'}(\gamma'_0 + n\cdot s' + h'_n,\; \eta'_0 +n \cdot t' + k'_n) \\
    & =  d_{S'}( \gamma'_n ,\; \eta'_n ) \\
    & =  d_{S'}\bigl( \varphi(\gamma_n) ,\; \varphi(\eta_n) \bigr) \\
    & =  d_S (\gamma_n ,\; \eta_n) \\
    & =  d_S (\gamma_0 + n\cdot s + h_n,\; \eta_0 +n \cdot s + k_n) \\
    & =  d_S (\gamma_0 - \eta_0,\; k_n - h_n), 
  \end{align*}
  and hence the triangle inequality (and translation invariance) yields
  \begin{align*}
    d_{S'} \bigl(n \cdot (s' - t'),\; \eta'_0 - \gamma'_0\bigr)
    & \leq d_S (\gamma_0 - \eta_0,\; k_n - h_n) 
    + d_{S'}(k'_n - h'_n,\; 0).
  \end{align*}
  Because $\tors G$ and $\tors G'$ are finite sets, the right hand side is 
  bounded independently of~$n$. Therefore, $\{ n \cdot (s' -t')\mid n \in \Z\}$ 
  lies in a $d_{S'}$-ball of finite radius around~$\eta'_0-\gamma'_0$. 
  Moreover, as $S'$ is finite, this ball is a finite set. 
  Thus, $s' - t' \in \tors G'$, as desired.
\end{proof}

\begin{prop}[parallel lines stay geodesics]\label{prop:parallelgeod}
  Let $G$ be a finitely generated Abelian group, let $S \subset G$ 
  be a symmetric finite generating set, and let $s \in S$.
  \begin{enumerate}
    \item\label{stm:geod} 
      If one quasi-algebraic line in~$\Cay GS$ of
      quasi-type~$\pi_G(s)$ is geodesic, then all quasi-algebraic
      lines in~$\Cay GS$ of quasi-type~$\pi_G(s)$ are geodesic lines.
    \item\label{stm:qconvex} 
      If one quasi-algebraic line in~$\Cay GS$ of
      quasi-type~$\pi_G(s)$ is a quasi-convex geodesic line, 
      then all quasi-algebraic lines in~$\Cay GS$ of 
      quasi-type~$\pi_G(s)$ are quasi-convex geodesic lines.
  \end{enumerate}
\end{prop}
\begin{proof}
  \emph{Ad~\ref{stm:geod}:} 
  \emph{Assume} for a contradiction that there exist 
  quasi-algebraic lines $\gamma, \eta \colon \Z \longrightarrow G$ 
  of quasi-type~$\pi_G(s)$ such that $\gamma$ is geodesic but $\eta$ is not. 
  Hence, there exist~$n,m \in \Z$ with $n \leq m$ and
  \[ d_S(\eta_n, \eta_m) \leq m - n -1. 
  \]
  Out of~$\eta$ we now construct the ``periodic'' quasi-algebraic
  line in~$\Cay GS$ of quasi-type~$\pi_G(s)$ given by
  \begin{align*}
    \widetilde \eta \colon \Z & \longrightarrow G \\
    z & \longmapsto \Bigl\lfloor \frac z{m-n}\Bigr\rfloor \cdot (\eta_m -\eta_n) 
                  + \eta_{n + z \mathbin{\;\text{mod}\;} (m-n)}.
  \end{align*}
  Because $\gamma$ and $\widetilde \eta$ both are of quasi-type~$\pi_G(s)$,
  we have $\gamma_k - \widetilde\eta_k \in \gamma_0 - \widetilde\eta_0 + \tors G$
  for all~$k \in \Z$. Adding the facts that $d_S$ is translation invariant 
  and that $\tors G$ is finite 
  we obtain 
  \[ d_S (\gamma_k, \widetilde\eta_k) 
     \leq d_S(\gamma_0, \widetilde\eta_0) + \diam_{d_S} \tors G
  \]
  for all~$k \in \Z$. By the assumption that $\gamma$ is geodesic and the 
  triangle inequality we therefore get 
  \begin{align*}
    k \cdot (m-n) 
    & = d_S(\gamma_0, \gamma_{k \cdot (m-n)}) \\
    & \leq d_S(\gamma_0, \widetilde\eta_0) 
         + d_S(\widetilde\eta_0, \widetilde\eta_{k \cdot (m-n)})
         + d_S(\widetilde\eta_{k \cdot (m-n)}, \gamma_{k \cdot (m-n)})\\
    & \leq d_S(\gamma_0,\widetilde\eta_0)
         + d_S(0, k \cdot (\eta_m - \eta_n)) 
         + d_S(\gamma_0, \widetilde\eta_0)
           + \diam_{d_S} \tors G\\
    & \leq k \cdot d_S(\eta_n, \eta_m) + 2 \cdot d_S(\gamma_0, \widetilde\eta_0)
           + \diam_{d_S} \tors G \\
    & \leq k \cdot (m-n-1) + 2 \cdot d_S(\gamma_0, \widetilde\eta_0)
           + \diam_{d_S} \tors G
  \end{align*}
  for all~$k \in \Z$, which leads to a contradiction for large
  enough~$k$. So if $\gamma$ is geodesic, then also $\eta$ must be
  geodesic.

  \emph{Ad~\ref{stm:qconvex}:} Any two quasi-algebraic lines 
  of the same quasi-type stay uniformly close (because the 
  metric~$d_S$ is translation invariant and $\tors G$ has 
  finite diameter). By definition, any geodesic line 
  uniformly close to a quasi-convex geodesic is itself quasi-convex. 
  Therefore, the second part is a consequence of the first part.
\end{proof}

\section{Proof of Cayley graph rigidity}\label{sec:proofthm}

Using the properties of quasi-algebraic and quasi-convex geodesic
lines established in the previous section, we now prove
Theorem~\ref{mainthm}:

\begin{proof}[Proof (of Theorem~\ref{mainthm})]
  Of course, we may assume that $S$ and $S'$ are symmetric; moreover, 
  because translations induce isomorphisms on the corresponding 
  Cayley graphs, we can assume that $\varphi(0) = 0$. We prove 
  the theorem by induction on~$S$.

  For the \emph{base case}, we suppose that $S$ consists of torsion
  elements (in particular, this contains the case~$S = \emptyset$
  and~$G=\{0\}$). In this case, $G = \tors G$, and the theorem clearly 
  holds.

  For the \emph{induction step} we now may assume that $S$ contains at
  least one non-torsion element and that the theorem holds for all
  subgroups generated by proper subsets of~$S$. We choose a basis
  of~$G /\tors G \cong \Z^{\rk G}$ and consider the corresponding
  $\ell^2$-norm on~$G/\tors G$. Let $s \in S$ be a
  $\|\pi_G(\cdot)\|_2$-maximal element of~$S$; because $S$ contains
  non-torsion elements, $s \not\in \tors G$.  Let us fix some
  notation: We write $s' := \varphi(s)$ and
  \begin{align*} 
    S(s) & := \bigl\{ t \in S 
             \bigm| \pi_G(t) = \pi_G(s) \bigr\},
             \\
     S'(s') & := \bigl\{ t' \in S' 
             \bigm| \pi_G(t') = \pi_G(s') \bigr\}, 
  \end{align*}
  as well as
  \[ S_s := S \setminus \bigl(S(s) \cup -S(s)\bigr)
     , \qquad
     S'_{s'} := S' \setminus \bigl(S'(s') \cup -S'(s')\bigr).
  \]
  We then let $G_s$ and $G'_s$ be the subgroups of~$G$ and~$G'$ 
  generated by~$S_s$ and $S'_{s'}$ respectively.
  
  We will now proceed in the following steps:
  \begin{enumerate}
    \item\label{st:lines} We show that $\varphi$ turns
      algebraic lines of type in~$S(s)$ into
      quasi-algebraic lines of quasi-type~$\pi_{G'}(s')$; in
      particular, we obtain for all~$y \in G$, all~$t \in S(s)$ and all~$m \in
      \Z$ that
      \[   \pi_{G'} \circ \varphi(y + m \cdot t) 
         = \pi_{G'} \circ \varphi(y) + m\cdot \pi_{G'}(s'). 
      \]
    \item\label{st:removes}
      We show that $\varphi$ induces an isomorphism~$\varphi_s$ between
      $\Cay{G_s}{S_s}$ and~$\Cay{G'_{s'}}{S'_{s'}}$.
    \item\label{st:ind} We apply the induction hypothesis to~$\varphi_s$.
    \item\label{st:additive} 
      Using the knowledge about $\varphi_s$ and algebraic 
      lines of type in~$S(s)$ we show that 
      \[ \pi_{G'} \circ \varphi \colon G \longrightarrow G'/\tors G' 
      \]
      is additive.
    \item \label{st:conclude}
      We conclude that the map
      \begin{align*}
        \overline\varphi \colon G/\tors G& \longrightarrow G'/\tors G'\\
        [g] & \longmapsto [\varphi(g)]
      \end{align*}
      is indeed well-defined and that $\overline \varphi$ is an affine 
      isomorphism.
  \end{enumerate}

  \emph{Ad~\ref{st:lines} (quasi-algebraic lines):} Because $s$ is
  $\|\pi_G(\cdot)\|_2$-maximal and $s \not\in \tors G$, quasi-algebraic lines
  of quasi-type~$\pi_G(s)$ in~$G$ are quasi-convex geodesic lines
  (Proposition~\ref{prop:qconvalg}~(1)). As $\varphi$ is a graph
  isomorphism, $\varphi$ therefore maps algebraic lines of any type
  in~$S(s)$ to quasi-algebraic quasi-convex geodesic lines
  (Proposition~\ref{prop:qconvalg}~(3)), which all are of the same
  quasi-type (Proposition~\ref{prop:paralleltype}). Looking at the
  algebraic line of type~$s$ through~$0$ shows that they all have 
  quasi-type~$\pi_{G'}(s')$ .

  \emph{Ad~\ref{st:removes} (reduced Cayley graphs):} We first show
  that $\varphi$ yields a graph isomorphism~$\Cay G{S_s}
  \longrightarrow \Cay{G'}{S_{s'}}$: In view of the first part it
  suffices to show that \emph{all} algebraic lines in~$G'$ of type
  in~$S'(s')$ are indeed $\varphi$-images of quasi-algebraic lines
  in~$G$ of quasi-type~$\pi_G(s)$. By the first step, at least one
  quasi-algebraic line in~$G'$ of quasi-type~$\pi_{G'}(s')$ is a
  quasi-convex geodesic line. Thus, by parallelism
  (Proposition~\ref{prop:parallelgeod}), all quasi-algebraic lines of
  quasi-type~$\pi_{G'}(s')$ in~$G'$ are quasi-convex geodesic lines.
  
  Let $\psi \colon \Cay{G'}{S'} \longrightarrow \Cay GS$ be the graph
  isomorphism inverse to~$\varphi$. The previous paragraph shows that
  we can apply Proposition~\ref{prop:qconvalg}~(3) to deduce that
  $\psi$ maps algebraic lines of type in~$S'(s')$ to
  quasi-algebraic lines in~$G$. Again, parallelism shows that these 
  quasi-algebraic lines in~$G$ must all be of quasi-type~$\pi_G(s)$ 
  (Proposition~\ref{prop:paralleltype}). 

  Therefore, $\varphi$ induces a graph isomorphism~$\Cay {G}{S_s}
  \longrightarrow \Cay{G'}{S'_{s'}}$. This graph isomorphism maps
  connected components to connected components; the connected
  component of~$0$ in~$\Cay G{S_s}$ is~$\Cay{G_s}{S_s}$, and the
  connected component of~$\varphi(0) = 0$ in~$\Cay{G'}{S'_{s'}}$
  is~$\Cay{G'_{s'}}{S'_{s'}}$. Hence, $\varphi$ induces a graph 
  isomorphism~$\Cay{G_s}{S_s} \longrightarrow \Cay{G'_{s'}}{S'_{s'}}$.

  \emph{Ad~\ref{st:ind} (induction):} By construction, $G_s$, $G'_{s'}$ are 
  finitely generated Abelian groups, generated by~$S_s$ and $S'_{s'}$ 
  respectively. By induction, the map
  \begin{align*}
    \overline\varphi_s \colon G_s/\tors G_s 
    & \longrightarrow G'_{s'} / \tors G'_{s'}\\
    [g] & \longmapsto [\varphi_s(g)]
  \end{align*}
  is well-defined and an affine isomorphism. Because of
  $\varphi(0)=0$, it follows that $\overline\varphi_s(0) = 0$ and so
  $\overline \varphi_s$ is a group isomorphism. In particular, also
  the map~$\pi_{G'_{s'}} \circ \varphi_s = \overline \varphi_s \circ
  \pi_{G_s}$ is a group homomorphism; hence, also $\pi_{G'} \circ \varphi_s$ is
  additive.

  \emph{Ad~\ref{st:additive} (additivity):} 
  Let $x, \widetilde x \in G$. We can write $x$ and $\widetilde x$ in 
  the form
  \[ x = y + \sum_{j=1}^k m_j \cdot s_j
     , \qquad
     \widetilde x 
     = \widetilde y 
     + \sum_{j=1}^{\widetilde k} \widetilde m_j \cdot \widetilde s_j,
  \]
  where $y, \widetilde y \in G_s$, and $k, \widetilde k \in \N$,
  $m_1, \dots, m_k, \widetilde m_1,
  \dots, \widetilde m_{\widetilde k} \in \Z$, as well 
  as $s_1, \dots, s_k, \widetilde s_1, \dots, \widetilde s_{\widetilde k} \in S(s)$. 
  In view of the first step we obtain
  \begin{align*}
    \pi_{G'} \circ \varphi (x) 
    & = \pi_{G'} \circ \varphi(y) + \sum_{j=1}^k m_j \cdot \pi_{G'}(s')
    ,\\
    \pi_{G'} \circ \varphi (\widetilde x) 
    & = \pi_{G'} \circ \varphi(\widetilde y) 
      + \sum_{j=1}^{\widetilde k} \widetilde m_j \cdot \pi_{G'}(s')
    ,\\
    \pi_{G'} \circ \varphi (x + \widetilde x) 
    & = \pi_{G'} \circ \varphi(y + \widetilde y) 
    + \sum_{j=1}^k m_j \cdot \pi_{G'}(s') 
    + \sum_{j=1}^{\widetilde k} \widetilde m_j \cdot \pi_{G'}(s').
  \end{align*}
  Because $y, \widetilde y \in G_s$ and because $\pi_{G'} \circ \varphi_s$ 
  is additive by the third step, we conclude that
  \begin{align*}
    \pi_{G'} \circ \varphi (x + \widetilde x) 
    & = \pi_{G'} \circ \varphi(y) + \pi_{G'}\circ \varphi(\widetilde y) 
    + \sum_{j=1}^k m_j \cdot \pi_{G'}(s') 
    + \sum_{j=1}^{\widetilde k} \widetilde m_j \cdot \pi_{G'}(s')
    \\
    & = \pi_{G'} \circ \varphi (x) + \pi_{G'} \circ \varphi(\widetilde x),
  \end{align*}
  as desired.

  \emph{Ad~\ref{st:conclude} (affine isomorphism):} By the previous
  step, $\pi_{G'} \circ \varphi \colon G \longrightarrow G'/\tors G'$
  is additive. Because $\varphi(0) = 0$ it follows that $\pi_{G'}
  \circ \varphi$ is a group homomorphism. Hence, $\pi_{G'} \circ \varphi$ 
  maps $\tors G$ to~$\tors(G' / \tors G') = \{0\}$. Therefore, 
  $\overline \varphi$ is well-defined, and a group homomorphism. 

  Applying the same arguments to the graph isomorphism inverse to~$\varphi$ 
  we see that the affine homomorphism~$\overline \varphi$ has an affine 
  inverse, and so is an affine isomorphism.

  This completes the proof of Theorem~\ref{mainthm}.
\end{proof}

The following example shows that in the theorem it is essential that
we divide out the torsion part: In general, not every isomorphism
between Cayley graphs of finitely generated Abelian groups is induced
from an affine isomorphism between the groups:

\begin{nonexa}\label{counterexample}
  We consider the graph~$\Cay {\Z \times \Z/2}{\{\pm(1,0), \pm(1,1)\}}$: 
  For any~$n \in \Z$ the flip at position~$n$, i.e., the 
  map
  \begin{align*}
    \Z \times\Z/2 & \longrightarrow \Z \times \Z/2 \\
    (x,y) & \longmapsto 
    \begin{cases}
      (x,y) & \text{if $x \neq n$} \\
      (x, 1-y) & \text{if $x = n$},
    \end{cases}
  \end{align*}
  is a graph automorphism of~$\Cay {\Z \times \Z/2}{\{\pm(1,0), \pm(1,1)\}}$
  (Figure~\ref{fig:flip}), 
  but this graph automorphism is not induced by an affine isomorphism 
  of~$\Z \times \Z/2$ (even though $\Z$ and $\Z/2$ are CI-groups). 
\end{nonexa}

\begin{figure}
  \begin{center}
    \begin{tikzpicture}[thick]
      \begin{scope}
      \clip (-2.3,-0.5) rectangle +(4.6,2);
      \draw (-3,0) -- (3,0);
      \draw (-3,1) -- (3,1);
      \draw (-3,1) -- (-2,0) -- (-1,1) -- (0,0) -- (1,1) -- (2,0) -- (3,1);
      \draw (-3,0) -- (-2,1) -- (-1,0) -- (0,1) -- (1,0) -- (2,1) -- (3,0);
      \gvertex{(-2,0)}
      \gvertex{(-1,0)}
      \gvertex{(0,0)}
      \gvertex{(1,0)}
      \gvertex{(2,0)}
      \gvertex{(-2,1)}
      \gvertex{(-1,1)}
      \gvertex{(0,1)}
      \gvertex{(1,1)}
      \gvertex{(2,1)}
      \end{scope}
      \begin{scope}[shift={(4.7,0)}]
        \draw (0,0) -- (1,0);
        \draw (0,1) -- (1,1);
        \draw[dashed] (0,0) -- (1,1);
        \draw[dashed] (0,1) -- (1,0);
        \gvertex{(0,0)}
        \gvertex{(1,0)}
        \gvertex{(0,1)}
        \gvertex{(1,1)}
        \draw (0,0) node[anchor=east] {$a\;$};
        \draw (1,0) node[anchor=west] {$\;b$};
        \draw (1,1) node[anchor=west] {$\;b'$};
        \draw (0,1) node[anchor=east] {$a'\;$};
        \draw (2.5,0.5) node {$\rightsquigarrow$};
        \draw (2.5,0.7) node[anchor=south] {\emph{flip}};
        \begin{scope}[shift={(4,0)}]
          \draw[dashed] (0,0) -- (1,0);
          \draw[dashed] (0,1) -- (1,1);
          \draw (0,0) -- (1,1);
          \draw (0,1) -- (1,0);
          \gvertex{(0,0)}
          \gvertex{(1,0)}
          \gvertex{(0,1)}
          \gvertex{(1,1)}
        \draw (0,0) node[anchor=east] {$a\;$};
        \draw (1,0) node[anchor=west] {$\;b'$};
        \draw (1,1) node[anchor=west] {$\;b$};
        \draw (0,1) node[anchor=east] {$a'\;$};
        \end{scope}
      \end{scope}
    \end{tikzpicture}
  \end{center}

  \caption{The graph $\Cay {\Z \times \Z/2}{\{\pm(1,0), \pm(1,1)\}}$ 
  and the effect of flipping}\label{fig:flip}
\end{figure}
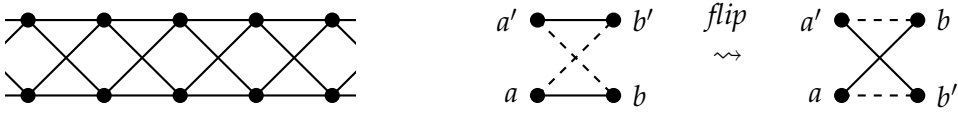

\section{Extracting the size of the torsion part}\label{sec:proofcor}

We will now deduce Corollary~\ref{maincor} from Theorem~\ref{mainthm}:

\begin{proof}[Proof (of Corollary~\ref{maincor})]
  Let $G$ and $G'$ be finitely generated Abelian groups with 
  $\rk G = \rk G'$ and $|\tors G| = |\tors G'|$. We write
  \[ d:= \rk G = \rk G'
     \quad\text{and}\quad
     k := |\tors G| = |\tors G'|
  .\]
  Let $S$ and $S'$ be $\Z$-bases of $G/\tors G \cong \Z^{d}$ and
  $G'/\tors G' \cong \Z^{d}$ respectively. Because of $G \cong \Z^d
  \times \tors G$ and $G' \cong \Z^d \times \tors G'$ we have
  \[ \Cay{G}{ S \cup \tors G}
     \cong \Cay {\Z^d} {\{e_1, \dots, e_d\}} \mathbin{\square} K_k
     \cong \Cay{G'}{S' \cup \tors G'},
  \]
  where $(e_1, \dots, e_d)$ is the standard basis of~$\Z^d$, where
  $K_k$ is the complete graph on $k$~vertices, and ``$\mathbin{\square}$''
  denotes the Cartesian product of graphs. So $G$ and $G'$ admit
  isomorphic Cayley graphs.

  Conversely, let $G$ and $G'$ be finitely generated Abelian groups
  that admit isomorphic Cayley graphs, i.e., there are finite
  generating sets~$S \subset G$, $S' \subset G'$ such that there is an
  isomorphism~$\varphi \colon \Cay GS \longrightarrow
  \Cay{G'}{S'}$. By Theorem~\ref{mainthm}, the graph
  isomorphism~$\varphi$ induces an affine isomorphism $\Z^{\rk G}
  \cong G /\tors G \longrightarrow G'/\tors G' \cong \Z^{\rk G'}$.

  Because affine isomorphisms between finitely generated Abelian free 
  groups are rank-preserving we obtain~$\rk G = \rk G'$.

  Moreover, the fact that $\varphi$ induces a well-defined map between 
  the quotients shows that $\varphi(\tors G) \subset \varphi(0) + \tors
  G'$. Because $\varphi$ is bijective it follows that
  \[ |\tors G| \leq |\tors G'|.
  \]
  Applying the same argument to~$\varphi^{-1}$ shows the reverse
  inequality, and hence we obtain $|\tors G| = |\tors G'|$.
\end{proof}

Of course, we cannot recover the complete algebraic structure of 
the torsion part -- there are finite Abelian groups of the same 
cardinality that are not isomorphic.

Another way to see that finitely generated Abelian groups that admit
isomorphic Cayley graphs have the same rank uses the growth rate: The
rank of a finitely generated Abelian group coincides with the growth
rate of the group~\cite[Chapter~VI]{delaharpe} and the growth rate is
preserved by graph isomorphisms.

\begin{problem}
  Can also the size of the torsion part of a finitely generated
  Abelian group be extracted from the growth function of every Cayley
  graph (with respect to a finite generating set)?
\end{problem}

In Section~\ref{sec:appx}, we will see that at least the parity of the 
torsion part can be detected in this way. However, the general 
case seems to be open.


\section[Detecting the parity of the torsion part]%
  {Detecting the parity of the torsion part,\\ alternative approach}
  \label{sec:appx}

The discussion on \textsf{mathoverflow.net}~\cite{mathoverflow} about
whether the relation ``admit isomorphic Cayley graphs'' is transitive
for finitely generated groups contains a neat argument by G.~Hainke
and L.~Scheele that allows to distinguish Cayley graphs of~$\Z$ from
Cayley graphs of~$\Z \times \Z/2$.  The idea is that for Abelian
groups taking inverses leads to automorphisms of Cayley graphs and
that counting fixed points of these automorphisms reveals the
$\Z/2$-factor.

In the following, we show how the same argument can be used to 
reveal the parity of the torsion part of finitely generated Abelian groups:  

\begin{prop}[parity of the torsion part via size of balls]
  Let $G$ be a finitely generated Abelian group and let $S \subset G$ 
  be a finite generating set. Moreover, let $r \in \N$ with 
  $r > \diam_{d_S} \tors G$. Then 
  \[ \beta(r) \equiv |\tors G| \mod 2,  
  \]
  where $\beta(r)$ is the number of elements of the $d_S$-ball~$B_S(r)$ of 
  radius~$r$ around~$0$ in~$G$.
\end{prop}
\begin{proof}
  Because $G$ is Abelian, the map
  \begin{align*}
    \tau \colon G & \longrightarrow G \\
    g & \longmapsto -g
  \end{align*}
  is an automorphism of~$\Cay GS$. Because $\tau (0 )= 0$ and because
  graph isomorphisms are isometries with respect to the graph metrics
  it follows that $\tau (B_S(r)) = B_S(r)$. Moreover, $\tau \circ \tau
  = \id_G$, and the fixed points of~$\tau$ are exactly the elements of~$G$ 
  of order at most~$2$; in particular, all fixed points of~$\tau$ lie 
  in~$\tors G$ and hence in~$B_S(r)$. Therefore, 
  \[ \beta(r) = |B_S(r)| \equiv \bigl|\{ g \in G \mid \tau(g) = g \}\bigr|
     \mod 2,
  \] 
  and thus
  \[ \beta(r)
     \equiv \bigl| \{ g \in G \mid \text{$g$ has order at most~$2$}\}\bigr| 
     \mod 2.
  \]
  Now counting the number of elements of order~$2$ in~$G$ finishes the proof 
  (see Lemma~\ref{lem:order2} below).
\end{proof}

\begin{lem}[number of elements of order at most~$2$]\label{lem:order2}
  Let $G$ be a finite Abelian group. Then 
  \[ \bigl| \{ g\in G \mid \text{$g$ has order at most~$2$}\}
     \bigr|
     \equiv 
     |G| 
     \mod 2.
  \]
\end{lem}
\begin{proof}
  Because $G$ is a finite Abelian group, we can write 
  \[ G \cong A \times \prod_{i \in I} B_i, 
  \]
  where $A$ is a finite Abelian group of odd order, $I$ is a finite
  set, and all~$B_i$ are finite cyclic groups of even order. An
  element of~$G$ has order at most~$2$ if and only if all of its
  components in the above product decomposition have order at
  most~$2$. The group~$A$ has exactly one element of order at most~$2$
  (namely, $0$), and each of the~$B_i$ has exactly two such elements 
  (namely, $0$ and the element corresponding to~$|B_i|/2$). Hence, 
  $G$ has exactly $2^{|I|}$~elements of order at most~$2$. Because 
  $|G| \equiv 2^{|I|} \mod 2$, the claim follows.
\end{proof}

\medskip
\vfill

\noindent
\emph{Clara L\"oh}\\
  {\small
  \begin{tabular}{@{\qquad}l}
    Fakult\"at f\"ur Mathematik\\
    Universit\"at Regensburg\\
    93040 Regensburg\\
    Germany\\
    \textsf{clara.loeh@mathematik.uni-regensburg.de}\\
    \textsf{http://www.mathematik.uni-regensburg.de/loeh}
  \end{tabular}}
\end{document}